\documentclass[11pt]{amsart}

\usepackage{graphicx}
\usepackage{latexsym}    
\usepackage{amssymb}   
\usepackage{amsmath}    
\usepackage{amsbsy}
\usepackage{amsthm}
\usepackage{amsgen}
\usepackage{amsfonts}
\usepackage{mathtools}
\usepackage{array}
\usepackage{mathrsfs}
\usepackage[all]{xy}    
\usepackage{tikz}
\usetikzlibrary{decorations.pathmorphing}
\usepackage{hyperref}
\usepackage{color}
\usepackage{verbatim}
\usepackage{url}
\usepackage{enumerate}
\usepackage{enumitem}

\usepackage{tabulary}

\usepackage{longtable}
\usepackage{rotating}
\usepackage{array}
\usepackage{multicol}
\usepackage{multirow}
\usepackage{makecell}
\usepackage{mathtools}
  \usepackage{cellspace}
\usepackage{booktabs}

\usepackage{multicol}
\usepackage{color}

\tikzset{snake it/.style={decorate, decoration=snake}}

\usepackage[normalem]{ulem}

\newtheorem{thm}{Theorem}[section]
\newtheorem{cor}[thm]{Corollary}
\newtheorem{lem}[thm]{Lemma}
\newtheorem{prop}[thm]{Proposition}

\newtheorem{theorem}{Theorem}

\newtheorem*{conjecture*}{Conjecture}
\newtheorem*{obs}{Observation}

\theoremstyle{definition}

\newtheorem*{defn*}{Definition}

\newtheorem*{conj*}{Conjecture}
\newtheorem{eg}[thm]{Example}

\newtheorem*{ack}{Acknowledgements}

\newenvironment{entry}%
{\begin{list}{}{%
\setlength{\labelwidth}{25pt}%
\setlength{\leftmargin}%
{\labelwidth+\labelsep}}}%
{\end{list}}

\newcommand{\DesLabel}[1]%
{\raisebox{-20pt}[1em][0pt]{%
\makebox[\labelwidth][l]%
{\parbox[t]{\labelwidth}{%
\hspace{4pt}\textbf{($\mathcal{#1}$)}}}}}
{\begin{entry}}%
{\end{entry}}

\numberwithin{equation}{section}

\DeclareMathOperator{\im}{Im}

\DeclareMathOperator{\Hom}{Hom}

\newcommand\lk{\operatorname{lk}}

\newcommand{\SO}{\mathrm{SO}}

\newcommand{\Pin}{\mathrm{Pin}}
\newcommand{\Pjn}{\mathrm{Pjn}}
\newcommand{\sph}{\mathbf{S}}

\newcommand{\Q}{\mathbb{Q}}

\def\ul{\underline}

\def\x{\times}
\def\ox{\otimes}

\def\lra{\longrightarrow}

\def\hra{\hookrightarrow}

\def\In{\subseteq}

\def\R{\mathbb{R}}
\def\Z{\mathbb{Z}}

\def\N{\mathbb{N}}
\def\DD{\mathbf{D}}

\def\<{\langle}
\def\>{\rangle}

\def\mc{\mathcal}

\def\ve{\varepsilon}

\def\bq{/\!\!/}
\def\Mab{M^7_{\ul{a}, \ul{b}}}
\def\Pab{P^{10}_{\ul{a}, \ul{b}}}

\def\Pina{\mathrm{Pin}(2)_{\ul{a}}}
\def\Pjnb{\mathrm{Pjn}(2)_{\ul{b}}}

\def\id{\mathrm{id}}

\def\Int{\mathrm{Int}}
\def\one{\mathbf 1}

\def\bpm{\begin{pmatrix}}
\def\epm{\end{pmatrix}}
\def\bvm{\begin{vmatrix}}
\def\evm{\end{vmatrix}}
\def\bsm{\left(\begin{smallmatrix}}
\def\esm{\end{smallmatrix}\right)}
\def\beq{\begin{equation}}
\def\eeq{\end{equation}}

\begin{document}



\title[The linking form and non-negative curvature]{Highly connected $7$-manifolds, the linking form and non-negative curvature}



\author[S.\ Goette]{S.\ Goette}
\address[Goette]{
Mathematisches Institut, Universit\"at Freiburg, Germany.}
\email{sebastian.goette@math.uni-freiburg.de}


\author[M.\ Kerin]{M.\ Kerin}
\address[Kerin]{School of Mathematics, Statistics and Applied Mathematics, N.U.I. Galway, Ireland.}
\email{martin.kerin@nuigalway.ie}


\author[K.\ Shankar]{K.\ Shankar}
\address[Shankar]{Department of Mathematics, University of Oklahoma, U.S.A.}
\email{Krishnan.Shankar-1@math.ou.edu}
\thanks{}

\date{\today}





\begin{abstract}
In a recent article, the authors constructed a six-parameter family of highly connected $7$-manifolds which admit an $\SO(3)$-invariant metric of non-negative sectional curvature.  Each member of this family is the total space of a Seifert fibration with generic fibre $\sph^3$ and, in particular, has the cohomology ring of an $\sph^3$-bundle over $\sph^4$.  In the present article, the linking form of these manifolds is computed and used to demonstrate that the family contains infinitely many manifolds which are not even homotopy equivalent to an $\sph^3$-bundle over $\sph^4$, the first time that any such spaces have been shown to admit non-negative sectional curvature.
\end{abstract}

\maketitle





Closed manifolds admitting non-negative sectional curvature are not very well understood and it is, at present, quite difficult to obtain examples with interesting topology.  This is partially explained by the dearth of known constructions, all of which depend in some way on two basic facts: First, compact Lie groups admit a bi-invariant metric (hence, non-negative curvature) and, second, Riemannian submersions do not decrease sectional curvature.  

In \cite{GKS1}, a $6$-parameter family of non-negatively curved, $2$-connected $7$-manifolds $\Mab$ was constructed, where the parameters $\ul a = (a_1, a_2, a_3), \ul b = (b_1, b_2, b_3) \in \Z^3$ satisfy $a_i, b_i \equiv 1 \!\! \mod 4$, for all $ i \in \{1,2,3\}$, and
$$
\gcd(a_1, a_2 \pm a_3) = 1 = \gcd(b_1, b_2 \pm b_3).
$$  
Each of the manifolds $\Mab$ is the total space of a Seifert fibration over an orbifold $\sph^4$ with generic fibre $\sph^3$ and has the cohomology  ring of an $\sph^3$-bundle over $\sph^4$.  In particular, $H^4(\Mab; \Z) = \Z_{|n|}$, where $n = \frac{1}{8} \det \bsm a_1^2 & b_1^2 \\ a_2^2 - a_3^2 & b_2^2 - b_3^2 \esm$ and, in the case $n = 0$, the notation $\Z_0$ signifies the integers $\Z$.  The manifolds $\Mab$ were shown in \cite{GKS1} to realise all exotic $7$-spheres.  To the authors' knowledge, this was the first time that it was observed that all exotic $7$-spheres are Seifert fibred by $\sph^3$.  The following result is somewhat surprising.


\begin{theorem}
\label{T:thmA}
Infinitely many of the manifolds $\Mab$ 
are not even homotopy equivalent to an $\sph^3$-bundle over $\sph^4$.
\end{theorem}

In the context of non-negative curvature, the construction of the manifolds $\Mab$ fits neatly into the general scheme of increasing topological complexity via reducing symmetry assumptions.  The standard example of a non-negatively curved manifold is a compact homogeneous space.  In \cite{GM}, Gromoll and Meyer discovered the first example of an exotic sphere admitting non-negative curvature by introducing the notion of a \emph{biquotient} $G \bq H$, that is, the quotient of a compact Lie group $G$ by a closed subgroup $H \In G \x G$ acting freely on $G$ via $(h_1, h_2) \cdot g = h_1 g h_2^{-1}$, $g \in G$, $(h_1, h_2) \in H$.  Clearly, the isometry group of a biquotient will, in general, be much smaller than that of a homogeneous space.  In contrast to the homogeneous situation, Totaro \cite{To} showed, for example, that there are infinitely many rational homotopy types of (non-negatively curved) biquotients already in dimension $6$.

An alternative approach to reducing symmetry is to assume that the manifold in question has low cohomogeneity, that is, that the quotient by a group of isometries is low dimensional.  In particular, when the quotient space is a closed interval, that is, for manifolds of \emph{cohomogeneity one}, Grove and Ziller \cite{GZ} discovered sufficient conditions to ensure the existence of an invariant metric of non-negative curvature, thus generalising earlier work of Cheeger \cite{Ch}, and used this to demonstrate that all $\sph^3$-bundles over $\sph^4$ admit a metric with non-negative curvature.

A cohomogeneity-one manifold as above naturally admits a codimension-one singular Riemannian foliation whose leaves are the orbits of the action, that is, are homogeneous spaces.  It was observed by Wilking in \cite{BWi} that a manifold which admits a codimension-one singular Riemannian foliation with biquotient leaves will also admit non-negative curvature, providing the sufficient conditions of Grove and Ziller \cite{GZ} are satisfied.   The 
manifolds $\Mab$ 
fall into this category and can thus be seen as a further success of the strategy of symmetry reduction.

The manifolds mentioned in Theorem \ref{T:thmA} occur in infinitely many cohomology types and are distinguished from $\sph^3$-bundles over $\sph^4$ by having a non-standard linking form.  In particular, these are the first manifolds with non-standard linking form observed to admit non-negative curvature (cf.\ \cite{GoKiSh}), thus implying that the linking form is not an obstruction to non-negative sectional curvature.

\begin{theorem}
\label{T:thmB}
Suppose the manifold $\Mab$ has $H^4(\Mab; \Z) = \Z_{|n|}$, $n \neq 0$.  Then there is a generator $\one \in H^4(\Mab; \Z)$ such that the linking form of $\Mab$ is given (up to sign) by
\begin{align*}
\lk : H^4(\Mab; \Z) \ox H^4(\Mab; \Z) \to& \Q/\Z \\
(x \one,y \one) \mapsto& 
 \pm \left(e_1 \, b_1^2 + e_0 \left(\tfrac{b_2^2 - b_3^2}{8} \right) \right) \frac{xy}{n} \!\! \mod 1 
\end{align*}
where $e_0, e_1 \in \Z$ satisfy 
$e_1 \, a_1^2 + e_0 \, \frac{1}{8}(a_2^2 - a_3^2) = 1$.
\end{theorem}

Observe that, if $f_0, f_1 \in \Z$ are chosen such that $f_1 \, b_1^2 + f_0 \, \frac{1}{8}(b_2^2 - b_3^2) = 1$, then
$$
\left( f_1 \, a_1^2 + f_0 \left(\tfrac{a_2^2 - a_3^2}{8} \right) \right) \left( e_1 \, b_1^2 + e_0 \left( \tfrac{b_2^2 - b_3^2}{8}\right) \right) \equiv 1 \!\! \mod n.
$$
Therefore, the linking form of $\Mab$ can equivalently be written (up to sign) as
$$
\lk(x \one',y \one') =  \pm \left(f_1 \, a_1^2 + f_0 \left( \tfrac{a_2^2 - a_3^2}{8} \right) \right) \frac{xy}{n}  \!\! \mod 1 
$$
with respect to the generator $\one' := \left(f_1 \, a_1^2 + f_0 \left( \tfrac{a_2^2 - a_3^2}{8} \right) \right) \one \in H^4(\Mab; \Z)$.

It will be demonstrated in Lemma \ref{L:stdLF} that $\Mab$ has standard linking form whenever $\gcd(a_1, b_1) = 1$.  In particular, this is the case for all $\sph^3$-bundles over $\sph^4$.  However, it is well known from \cite{DWi}  that there exist $2$-connected $7$-manifolds with non-standard linking form which have the same cohomology ring as in the case $\gcd(a_1, b_1) = 1$: see, for instance, Example \ref{Eg:nonst}.  

\begin{obs}
The manifolds $\Mab$ do not realise all $2$-connected $7$-manifolds with the cohomolgy ring of an $\sph^3$-bundle over $\sph^4$.  
\end{obs}

In light of this observation, it is tempting to make the following conjecture.

\begin{conjecture*}
Every $2$-connected $7$-manifold with the cohomology ring of an $\sph^3$-bundle over $\sph^4$ admits a non-negatively curved, codimension-one singular Riemannian foliation with singular leaves of codimension two, and a Seifert fibration onto an orbifold $\sph^4$ with generic fibre $\sph^3$.
\end{conjecture*}

The paper is organised as follows.  In Section \ref{S:prelim}, the construction and properties of the manifolds $\Mab$ are reviewed and relevant notation introduced, before the linking form is introduced and some important facts recalled.  In Section \ref{S:Bockstein}, the structure of the manifolds $\Mab$ is used to obtain an understanding of the Bockstein homomorphism.  Theorem \ref{T:thmB} is proved in Section \ref{S:linking}, while Section \ref{S:numth} is dedicated to the elementary number theory necessary to construct explicit manifolds $\Mab$ satisfying the conclusion of Theorem \ref{T:thmA}.


\ack{It is a pleasure to thank Diarmuid Crowley for his interest in this project and for useful conversations about the linking form.  Part of this research was performed at the mathematical research institute MATRIX in Australia and the authors wish to thank the institute for its hospitality.  S.\ Goette and M.\ Kerin have received support from the DFG Priority Program 2026 \emph{Geometry at Infinity}, while M.\ Kerin and K.\ Shankar received support from SFB 898: \emph{Groups, Geometry \& Actions} at WWU M\"unster.  K.\ Shankar received support from the National Science Foundation.\footnote{The views expressed in this paper are those of the authors and do not necessarily reflect the views of the National Science Foundation.}}


\section{Preliminaries and notation}
\label{S:prelim}


\subsection{The manifolds $\Mab$\, } \hspace*{1mm}\\
\label{SS:family}

Suppose that a compact Lie group $G$ acts smoothly on a closed, connected, smooth manifold $M$ via $G \x M \to M,\ (g,p) \mapsto g \cdot p$.  For each $p \in M$, the \emph{isotropy group} at $p$ is the subgroup $G_p = \{g \in G \mid g \cdot p = p\} \In G$, and the \emph{orbit} through $p$ is the submanifold $G \cdot p = \{g \cdot p \in M \mid g \in G \} \In M$.  The manifold $M$ is foliated by $G$-orbits and an orbit $G \cdot p$ is diffeomorphic to the homogeneous space $G/G_p$.

The action $G \x M \to M$ is said to be of \emph{cohomogeneity one} if there is an orbit of codimension one or, equivalently, if $\dim(M/G) = 1$.  In such a case, the manifold $M$ is called a \emph{cohomogeneity-one ($G$-)manifold}.  If, in addition, $\pi_1(M)$ is assumed to be finite, then the orbit space $M/G$ can be identified with a closed interval.  By fixing an appropriately normalised $G$-invariant metric on $M$, it may be assumed that $M/G = [-1,1]$.  Let $\pi: M \to M/G = [-1,1]$ denote the quotient map.  The orbits $\pi^{-1}(t)$, $t \in (-1,1)$, are called \emph{principal orbits} and the orbits $\pi^{-1}(\pm 1)$ are called \emph{singular orbits}.

Choose a point $p_0 \in \pi^{-1}(0)$ and consider a geodesic $c:\R \to M$ orthogonal to all the orbits, such that $c(0) = p_0$ and $\pi \circ c|_{[-1,1]} = \id_{[-1,1]}$.  Then, for every $t \in (-1,1)$, one has $G_{c(t)} = G_{p_0} \In G$, and this \emph{principal isotropy group} will be denoted by $H \In G$.  If $p_\pm = c(\pm 1) \in M$, denote the \emph{singular isotropy groups} $G_{p_\pm}$ by $K_\pm$ respectively.  In particular, $H \In K_\pm$.

By the slice theorem, $M$ can be decomposed as the union of two disk-bundles, over the singular orbits $G/K_- = \pi^{-1}(- 1)$ and $G/K_+ = \pi^{-1}(+ 1)$ respectively, which are glued along their common boundary $G/H = \pi^{-1}(0)$:
$$
M = (G \x_{K_-} \DD^{l_-}) \cup_{G/H} (G \x_{K_+} \DD^{l_+}) \, .
$$
Since the principal orbit $G/H$ is the boundary of both disk-bundles, it follows that $K_\pm/H = \sph^{l_\pm-1}$, where $l_\pm$ denote the codimensions of $G/K_\pm \In M$.

Conversely, given any chain $H \In K_\pm \In G$, with $K_\pm/H = \sph^{d_\pm}$, one can construct a cohomogeneity-one $G$-manifold $M$ with codimension $d_\pm + 1$ singular orbits.  For this reason, a cohomogeneity-one manifold is conveniently represented by its group diagram:
$$
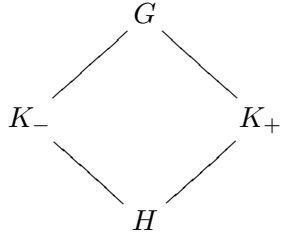
\xymatrix{
& G & \\
K_- \ar@{-}[ur] & & K_+ \ar@{-}[ul] \\
&  H \ar@{-}[ur] \ar@{-}[ul] & 
}
$$

In \cite{GZ}, the authors determined a sufficient condition for a cohomogeneity-one manifold to admit non-negative curvature.

\begin{thm}[\cite{GZ}]
\label{T:GZ}
Let $G$ be a compact Lie group acting on a manifold $M$ with cohomogeneity one.  If the singular orbits are of codimension $2$, then $M$ admits a $G$-invariant metric of non-negative sectional curvature.
\end{thm}

Consider now the subgroups  
\begin{align*}
Q &= \{\pm 1, \pm i, \pm j, \pm k\}, \\
\Pin(2) &= \{e^{i \theta} \mid \theta \in \R\} \cup \{e^{i \theta} j \mid \theta \in \R\},  \\
\Pjn(2) &= \{e^{j \theta} \mid \theta \in \R\} \cup \{i \, e^{j \theta} \mid \theta \in \R\}
\end{align*}
of the group $\sph^3$ of unit quaternions, where the notation $\Pjn(2)$ is intended to be suggestive since, clearly, the groups $\Pin(2)$ and $\Pjn(2)$ are isomorphic, the only difference being that the roles of $i$ and $j$ are switched.

For $\ul a = (1, a_2, a_3), \ul b = (1, b_2, b_3) \in \Z^3$, with $a_i, b_i \equiv 1$ mod $4$ for all $i \in \{1,2,3\}$ and $\gcd(a_1, a_2, a_3) = \gcd(b_1, b_2, b_3) = 1$, a family of cohomogeneity-one ($\sph^3 \x \sph^3 \x \sph^3$)-manifolds $\Pab$ was introduced in \cite{GKS1} via the group diagram
\beq
\label{E:Pab}
\xymatrix{
& \sph^3 \x \sph^3 \x \sph^3 & \\
\Pina \ar@{-}[ur] & & \Pjnb \ar@{-}[ul] \\
&  \Delta Q \ar@{-}[ur] \ar@{-}[ul] & 
}
\eeq
where the principal isotropy group $\Delta Q$ denotes the diagonal embedding of $Q$ into $\sph^3 \x \sph^3 \x \sph^3$, and the singular isotropy groups are given by
\begin{align*}
\Pina &= \{(e^{i a_1 \theta}, e^{i a_2 \theta}, e^{i a_3 \theta}) \mid \theta \in \R\} \cup \{(e^{i a_1 \theta} j, e^{i a_2 \theta} j, e^{i a_3 \theta} j) \mid \theta \in \R\},  \\
\Pjnb &= \{(e^{j b_1 \theta}, e^{j b_2 \theta}, e^{j b_3 \theta}) \mid \theta \in \R\} \cup \{(i \, e^{j b_1 \theta}, i\, e^{j b_2 \theta}, i\, e^{j b_3 \theta}) \mid \theta \in \R\}.
\end{align*}
Note that the restriction $a_i, b_i \equiv 1$ mod $4$ is to ensure only that $\Delta Q$ is a subgroup of both $\Pina$ and $\Pjnb$.  The subfamily consisting of those $\Pab$ having $a_1 = b_1 = 1$ describes all principal ($\sph^3 \x \sph^3$)-bundles over $\sph^4$; see \cite{GZ}.

For the sake of notation, let $G = \sph^3 \x \sph^3 \x \sph^3$ from now on.  It was proven in \cite[Lemma 1.2]{GKS1} that the subgroup $\{1\} \x \Delta \sph^3 \In \{1\} \x \sph^3 \x \sph^3 \In G$ acts freely on $\Pab$ if and only if 
\beq
\label{E:free}
\gcd(a_1, a_2 \pm a_3) = 1 \ \textrm{ and } \ \gcd(b_1, b_2 \pm b_3) = 1.
\eeq

Therefore, given a cohomogeneity-one $G$-manifold $\Pab$ determined by a group diagram \eqref{E:Pab} satisfying the conditions \eqref{E:free}, one obtains a smooth, $7$-dimensional manifold $\Mab$ defined via
$$
\Mab = (\{1\} \x \Delta \sph^3) \backslash \Pab \, .
$$
Since the singular orbits of the cohomogeneity-one $G$-action on $\Pab$ are of codimension $2$, it follows from Theorem \ref{T:GZ} that each $\Pab$ admits a $G$-invariant metric of non-negative sectional curvature.  As the free action of $\{1\} \x \Delta \sph^3$ is by isometries, there is an induced metric of non-negative curvature on $\Mab$.

By construction, there is a codimension-one singular Riemannian foliation of $\Mab$ by biquotients, such that the leaf space is $[-1,1]$ and $\Mab$ decomposes as a union of two-dimensional disk-bundles over the two singular leaves which are glued along their common boundary, a regular leaf.  This follows easily from the Slice Theorem applied to $\Pab$.  Indeed, the action of $\{1\} \x \Delta \sph^3$ preserves the $G$-orbits of $\Pab$, and the image of an orbit $G/U$ is a leaf given by
\beq
\label{E:BiqDiff}
(\{1\} \x \Delta \sph^3) \backslash G / U \cong (\sph^3 \x \sph^3) \bq U \, ,
\eeq
where this diffeomorphism is induced by 
$$
(q_1 \, u_1, q_2 \, u_2, q_3 \, u_3) \mapsto (q_1 \, u_1, u_2^{-1} q_2^{-1} q_3 \, u_3),
$$ 
for $(q_1, q_2, q_3) \in G$ and $(u_1, u_2, u_3) \in U \In G$. Viewing $\Mab$ in this way, the $\gcd$ conditions \eqref{E:free} required in the definition are simply the conditions ensuring that each of the biquotient actions on $\sph^3 \x \sph^3$ is free.

If $\ve \in (-1,1)$ and if $\tau : \Mab \to [-1,1]$ denotes the projection onto the leaf space of the codimension-one foliation of $\Mab$ by biquotients, define  
$$
M_- = \tau^{-1}([-1,\ve)), \ 
M_+ = \tau^{-1}((-\ve, 1]) \ 
\text{ and } \ 
M_0 = \tau^{-1}(-\ve, \ve).
$$
The preimages $M_\pm$ are two-dimensional disk-bundles over the singular leaves $(\sph^3 \x \sph^3) \bq \Pina$ and  $(\sph^3 \x \sph^3) \bq \Pjnb$, while $M_0  = M_- \cap M_+ \cong (\sph^3 \x \sph^3) \bq \Delta Q \x (-\ve,\ve) $.  Clearly $\Mab = M_- \cup M_+$.

It was shown in \cite{GKS1} that the manifolds $\Mab$ are $2$-connected and that 
\beq
\label{E:cohom}
H^4(\Mab; \Z) = \Z_{|n|},
\text{ where }
n = \frac{1}{8} \det \bpm a_1^2 &  b_1^2 \\ a_2^2 - a_3^2 & b_2^2 - b_3^2 \epm.
\eeq  
The notation $\Z_0$ signifies the integers $\Z$, in the case $n = 0$.  From Lemmas 2.6 and 2.7 of \cite{GKS1} it follows that
\beq
\label{E:leafcohom}
\begin{split}
H^j(M_\pm ; \Z) &= 
\begin{cases}
\Z, & j = 0,3, \\
\Z_2, & j = 2, 5, \\
0, & \text{otherwise,}
\end{cases} \\
H^j(M_0 ; \Z) &= 
\begin{cases}
\Z, & j = 0,6, \\
\Z_2 \oplus \Z_2, & j = 2, 5,\\
\Z \oplus \Z, & j = 3,\\
0, & \text{otherwise.}
\end{cases}
\end{split}
\eeq

Denote by 
\beq
\label{E:inclusions}
i_\pm : M_\pm \hra \Mab 
\ \text{ and } \ 
j_\pm : M_0 \hra M_\pm
\eeq 
the respective inclusion maps, and by 
\beq
\label{E:pairmaps}
q_\pm : (\Mab, \emptyset) \to (\Mab, M_\pm) 
\ \text{ and } \ 
f_\pm : (M_\mp, M_0) \to (\Mab, M_\pm)
\eeq
the maps of pairs induced by the identity map on $\Mab$ and by $i_\pm$ respectively.  Note, furthermore, that the maps on cohomology induced by the inclusions $j_\pm$ are determined by the projection maps $\pi_\pm$ in the circle-bundles 
\begin{align*}
\sph^1 = \Pina/\Delta Q &\lra (\sph^3 \x \sph^3) \bq \Delta Q \stackrel{\pi_-}{\lra} (\sph^3 \x \sph^3) \bq \Pina \,, \\[1mm]
\sph^1 = \Pjnb/\Delta Q &\lra (\sph^3 \x \sph^3) \bq \Delta Q \stackrel{\pi_+}{\lra} (\sph^3 \x \sph^3) \bq \Pjnb \, ,
\end{align*}
since $\pi_\pm$ respect deformation retractions of $M_-$, $M_+$ and $M_0$ onto the respective leaves.  In particular, the maps $\pi_\pm^*$ been computed in degree three in \cite[Equation (2.16)]{GKS1} and, with respect to fixed bases $\{x_\pm\}$ of $H^3(M_\pm ; \Z) = \Z$ and  $\{v_1, v_2\}$ of $H^3(M_0; \Z) = \Z \oplus \Z $, yield
\beq
\label{E:maps}
\begin{split}
j_-^*(x_-) &= \frac{1}{8}(a_2^2 - a_3^2) \, v_1 + a_1^2 \, v_2, \\
j_+^*(x_+) &= -\frac{1}{8}(b_2^2 - b_3^2) \, v_1 - b_1^2 \, v_2, 
\end{split}
\eeq
which, by the gcd conditions \eqref{E:free}, are each generators of $H^3(M_0; \Z)$.  Finally, by excision, the induced homomorphisms $f_\pm^* : H^j(\Mab, M_\pm; R) \to H^j(M_\mp, M_0; R)$ are isomorphisms in all degrees, for any choice of coefficient ring $R$.

\medskip


\subsection{The linking form} \hspace*{1mm}\\
\label{SS:link}

If $M$ is an $(s-1)$-connected, closed, oriented, smooth, $(2s+1)$-dimensional manifold, let $TH_s(M)$ and $TH^{s+1}(M;\Z)$ denote the torsion subgroups of respective integral homology and cohomology groups.  Let $a \in \mc C_{s}(M)$ be a chain representing a homology class $[a] \in TH_{s}(M)$.  Then there is some $n_a \in \Z$ such that $n_a \cdot [a] = 0$ and, hence, some $c_a \in \mc C_{s+1} (M)$ such that $n_a \cdot a$ is the boundary of $c_a$, that is, $n_a \cdot a = \partial c_a$.  The \emph{linking form} is a non-degenerate, bilinear pairing defined by
\beq
\label{E:LFhom}
\begin{split}
\lk : TH_{s}(N) \ox TH_{s}(M) &\to \Q / \Z \\
([a], [b]) &\mapsto \frac{\Int(c_a, b)}{n_a} \mod 1,
\end{split}
\eeq
where $\Int: \mc C_{s+1}(M) \x \mc C_{s}(M) \to \Z$ yields the signed count of intersections of its arguments with respect to the orientation of $M$.  The linking form is symmetric (respectively, skew-symmetric) for $s$ odd (respectively, $s$ even).  It was introduced in \cite{Br} and \cite{ST}.

Consider now the short exact sequence 
$$
0 \lra \Z \stackrel{m}{\lra} \Q \stackrel{r}{\lra} \Q / \Z \lra 0.
$$
The boundary homomorphism $\beta : H^j(M; \Q/\Z) \to H^{j+1}(M; \Z)$ in the associated long exact sequence 
$$
\dots \lra H^{j}(M; \Z) 
\stackrel{m}{\lra} H^{j}(M; \Q) 
\stackrel{r}{\lra} H^{j}(M; \Q/\Z) 
\stackrel{\beta}{\lra}H^{j+1}(M; \Z) 
\lra \dots
$$
is called the \emph{Bockstein homomorphism}.  Observe that $TH^j(M; \Z) \In \im(\beta)$, since $TH^j(M; \Z)$ lies in the kernel of $m : H^{j}(M; \Z) \to H^{j}(M; \Q) $.  

Now, if $D : H_j(M) \to H^{2s + 1 - j}(M; \Z)$ denotes the inverse of Poincar\'e duality, $[M] \in H_{2s+1}(M)$ the fundamental class of $M$ and $\< \, ,\> : H^j(M;R) \ox H_j(M) \to R$ the $R$-valued Kronecker pairing, the right-hand side of \eqref{E:LFhom} is given, modulo the integers, by
$$
\frac{\Int(c_a,b)}{n_a} = \< w_a \smile D([b]), [M] \>,
$$
where $w_a \in H^{s}(M; \Q/\Z)$ is such that $\beta(w_a) = D([a])$.  That is, the linking form can be rewritten as a non-degenerate, bilinear form 
\beq
\label{E:LFcohom}
\begin{split}
\lk : TH^{s+1}(M; \Z) \ox TH^{s+1}(M; \Z) &\to \Q/\Z \\
(x,y) &\mapsto \< w \smile y, [M] \> \mod 1,
\end{split}
\eeq
where $\beta(w) = x \in TH^{s+1}(M; \Z)$.  Note, in particular, that the sign of the linking form depends on the choice of orientation on $M$.  Furthermore, if $H^{s+1}(M; \Z)$ is torsion, that is, $TH^{s+1}(M; \Z) = H^{s+1}(M; \Z)$, then $M$ being $(s-1)$-connected implies that the Bockstein homomorphism is an isomorphism and it follows from \eqref{E:LFcohom} that 
\beq
\label{E:LF}
\lk(x,y) = \<\beta^{-1}(x) \smile y, [M]\>,
\eeq
for all $x,y \in H^{s+1}(M; \Z)$.

Suppose now that $TH^{s+1}(M; \Z)$ is cyclic of order $n$.  In this case, bilinearity ensures that the linking form is completely determined by $\lk(\one, \one)$, where $\one$ is some generator of $TH^{s+1}(M; \Z) = \Z_n$.  The linking form is said to be \emph{standard} if there exists an isomorphism $\theta : TH^{s+1}(M; \Z) \to TH^{s+1}(M; \Z)$ such that 
$$
\lk(\theta(\one), \theta(\one)) = \frac{1}{n} \in \Q/\Z.
$$
Recall, however, that the group of isomorphisms of $\Z_n$ is isomorphic to the group of units $\Z_n^* \In \Z_n$.  Therefore, the linking form is standard if and only if there is some unit $\lambda \in \Z_n^*$ such that
$$
\lk(\one, \one) = \frac{\lambda^2}{n} \mod 1.
$$
For $2$-connected $7$-manifolds, the linking form being standard imposes topological restrictions on the manifold.

\begin{thm}[{\cite[Corollary 2]{KiSh}}]
\label{T:KS}
A closed, smooth, $2$-connected $7$-manifold $M$, with $H^4(M; \Z)$ finite cyclic, is homotopy equivalent to an $\sph^3$-bundle over $\sph^4$ if and only if its linking form is standard for some choice of orientation on $M$.
\end{thm}

By work of Crowley and Escher \cite{CE} and Kitchloo and Shankar \cite{KiSh} (cf.\ \cite{DWi}), the homotopy equivalence in Theorem \ref{T:KS} can, in fact, be strengthened to equivalence under a PL-homeomorphism.

The strategy for proving Theorem \ref{T:thmA} is now clear.  One must identify manifolds $\Mab$ which have a non-standard linking form, regardless of the choice of orientation.  A simple example might shed some light on the number theoretic side of the problem, even though, by Lemma \ref{L:stdLF}, this particular example cannot occur among the manifolds $\Mab$.

\begin{eg}
\label{Eg:nonst}
Suppose $M$ is a closed, smooth, $2$-connected $7$-manifold with $H^4(M; \Z) = \Z_5$ and $\lk(\one, \one) = \frac{2}{5} \in \Q/\Z$, for some generator $\one \in H^4(M; \Z) $.  Since $\pm 2 \in \Z_5$ is not the square of a unit in $\Z_5^* = \{1,2,3,4\}$, it follows from Theorem \ref{T:KS} that $M$ is not homotopy equivalent to an $\sph^3$-bundle over $\sph^4$. 
\end{eg}

A well-known fact from the study of quadratic reciprocities is being exploited in Example \ref{Eg:nonst} and plays an important role in finding further examples; namely, for an odd prime $p$, the unit $-1 \in \Z_p^*$ is a square if and only if $p \equiv 1$ mod $4$.  Since the squares make up only half of all units in $\Z_p$, this implies that multiplying a non-square by $-1$ will not turn it into a square whenever $p \equiv 1$ mod $4$.  This observation will yield non-standard linking forms, even up to a change of sign.



\section{The Bockstein homomorphism}
\label{S:Bockstein}

Since the formula \eqref{E:LF} describes the linking form of the manifolds $\Mab$, it will be important in what follows to have a good understanding of the Bockstein homomorphism for these manifolds.

Recall that $\Mab = M_- \cup M_+$ and $M_- \cap M_+ = M_0$ and suppose from now on that 
\beq
\label{E:finite}
H^4(\Mab; \Z) = \Z_{|n|},
\text{ where }
n = \frac{1}{8} \det \bpm a_1^2 &  b_1^2 \\ a_2^2 - a_3^2 & b_2^2 - b_3^2 \epm \neq 0.
\eeq

It follows from \eqref{E:leafcohom} and the long exact cohomology sequence for the pair $(\Mab, M_\pm)$ that
$$
H^3(\Mab, M_\pm ; \Z) = \Z_2
\ \text{ and } \ 
H^5(\Mab, M_\pm ; \Z) = 0.
$$
On the other hand, the long exact sequence for the pair $(M_\mp, M_0)$ yields a short exact sequence
$$
0 \lra H^3(M_\mp; \Z) \stackrel{j_\mp^*}{\lra }
H^3(M_0; \Z) \lra 
H^4(M_\mp, M_0; \Z) \lra 0.
$$
By \eqref{E:leafcohom} and \eqref{E:maps}, it now follows that $H^4(M_\mp, M_0; \Z) = \Z$.  However, excision implies that the map $f_\pm^* : H^4(\Mab, M_\pm; \Z) \to H^4(M_\mp, M_0; \Z)$ is an isomorphism, from which it may be concluded that
$$
H^4(\Mab, M_\pm; \Z) = \Z.
$$

These considerations, together with the Universal Coefficient Theorem for cohomology \cite[Chap.\ 5, Theorem 10]{Sp}, now easily yield the cohomology groups listed in Table \ref{table:cohomgps}.

\begin{table}

\begin{tabular}[h]{|Sc||Sc|c|Sc|}
\cline{2-4}
\multicolumn{1}{c||}{} & $\Mab$ & $M_\pm$ & $(\Mab, M_\pm)$ 
\\ \hline \hline
$H^3(- , ; \Z)$ & $0$ & $\Z$ & $\Z_2$ \\ \hline
$H^3(-; \Q)$ & $0$ & $\Q$ & $0$ \\ \hline
$H^3(-; \Q/\Z)$ & $\Z_{|n|}$ & $\Q/\Z$ & $0$
\\ \hline\hline
$H^4(-; \Z)$ & $\Z_{|n|}$ & $0$ & $\Z$ \\ \hline
$H^4(-; \Q)$ & $0$ & $0$ & $\Q$ \\ \hline
$H^4(-; \Q/\Z)$ & $0$ & $\Z_2$ & $\Q/\Z$ \\ \hline
\end{tabular}

\vspace{10pt}
\caption{Important cohomology groups in $\Z$, $\Q$ and $\Q/\Z$ coefficients.}
\label{table:cohomgps}
\end{table}


From the short exact coefficient sequence
$$
0 \lra \Z \stackrel{m}{\lra} \Q \stackrel{r}{\lra} \Q / \Z \lra 0,
$$
together with the maps $i_\pm : M_\pm \hra \Mab$ and $q_\pm : (\Mab, \emptyset) \to (\Mab, M_\pm)$, one obtains a commutative diagram 
\beq
\label{E:bigdiag}
\xymatrix@C=0.54cm{
& 0 \ar[d] & 0 \ar[d] & 0 \ar[d]  & \\
0 \ar[r] & \mc C^*(\Mab, M_\pm; \Z) \ar[r]^(0.55){q_\pm^*} \ar[d]^m & 
\mc C^*(\Mab; \Z) \ar[r]^{i_\pm^*} \ar[d]^m &
\mc C^*(M_\pm; \Z) \ar[r] \ar[d]^m & 0 \\
0 \ar[r] & \mc C^*(\Mab, M_\pm; \Q) \ar[r]^(0.55){q_\pm^*} \ar[d]^r & 
\mc C^*(\Mab; \Q) \ar[r]^{i_\pm^*} \ar[d]^r &
\mc C^*(M_\pm; \Q) \ar[r] \ar[d]^r & 0 \\
0 \ar[r] & \mc C^*(\Mab, M_\pm; \Q/\Z) \ar[r]^(0.55){q_\pm^*} \ar[d] & 
\mc C^*(\Mab; \Q/\Z) \ar[r]^{i_\pm^*} \ar[d] &
\mc C^*(M_\pm; \Q/\Z) \ar[r] \ar[d] & 0 \\
& 0 & 0 & 0 &
}
\eeq
of cochain complexes.  This induces a commutative diagram 

\beq
\label{E:cohomdiag}
\resizebox{\displaywidth}{!}{\xymatrix@C=0.4cm{
&  & H^3(\Mab, M_\pm; \Q/\Z) \ar[d] \ar[r]^(0.55){q_\pm^*} & H^3(\Mab; \Q/\Z) \ar[d]^\beta \\
H^3(\Mab; \Z) \ar[r]^{i_\pm^*} \ar[d]^m & 
H^3(M_\pm; \Z) \ar[r]^(0.45){\delta_\pm} \ar[d]^m &
H^4(\Mab, M_\pm; \Z) \ar[r]^(0.55){q_\pm^*} \ar[d]^m &
H^4(\Mab; \Z) \ar[d]^m \\
H^3(\Mab; \Q) \ar[r]^{i_\pm^*} \ar[d]^r & 
H^3(M_\pm; \Q) \ar[r]^(0.45){\delta_\pm} \ar[d]^r &
H^4(\Mab, M_\pm; \Q) \ar[r]^(0.55){q_\pm^*} \ar[d]^r &
H^4(\Mab; \Q) \ar[d]^r \\
H^3(\Mab; \Q/\Z) \ar[r]^{i_\pm^*} \ar[d]^\beta & 
H^3(M_\pm; \Q/\Z) \ar[r]^(0.45){\delta_\pm} \ar[d] &
H^4(\Mab, M_\pm; \Q/\Z) \ar[r]^(0.55){q_\pm^*} &
H^4(\Mab; \Q/\Z) \\
H^4(\Mab; \Z) \ar[r]^{i_\pm^*}  & 
H^4(M_\pm; \Z) & &
}}
\eeq
of long exact sequences for the pair $(\Mab, M_\pm)$, where $\delta_\pm : H^j(M_\pm; R) \to H^{j+1}(\Mab, M_\pm; R)$, $R \in \{\Z, \Q, \Q/\Z\}$, denotes the coboundary homomorphism.  

By exactness and by Table \ref{table:cohomgps}, it can immediately be deduced from diagram \eqref{E:cohomdiag}: both the Bockstein homomorphism $\beta : H^3(\Mab; \Q/\Z) \to H^4(\Mab;\Z)$ and $\delta_\pm : H^3(M_\pm; \Q) \to H^{4}(\Mab, M_\pm; \Q)$ are isomorphisms; the homomorphisms $m : H^4(\Mab, M_\pm; \Z) \to H^4(\Mab, M_\pm; \Q)$ and $i_\pm^* : H^3(\Mab; \Q/ \Z) \to H^3(M_\pm; \Q/\Z)$ are injective; and the homomorphims $r : H^3(M_\pm; \Q) \to H^3(M_\pm; \Q/\Z)$ and $q_\pm^* : H^4(\Mab, M_\pm; \Z) \to H^4(\Mab; \Z)$ are surjective.  From these observations, it is now possible to gain some further understanding of the Bockstein homomorphism.

\begin{prop}
\label{P:Bockstein}
Suppose $\Mab$ satisfies \eqref{E:finite}.  Then, with the notation above, the Bockstein homomorphism  satisfies
$$
i_\pm^* \circ \beta^{-1} \circ q_\pm^* = r \circ \delta_\pm^{-1} \circ m : H^4(\Mab, M_\pm; \Z) \to H^3(M_\pm; \Q/\Z).
$$
\end{prop}

\begin{proof}
The proof will be on the level of cochains.  If $w_\pm \in \mc C^4(\Mab, M_\pm; Z)$ represents a cohomology class $[w_\pm] \in H^4(\Mab, M_\pm; \Z)$, then, since the Bockstein homomorphism  $\beta : H^3(\Mab; \Q/\Z) \to H^4(\Mab;\Z)$ is an isomorphism, there is a unique class $[w] \in H^3(\Mab; \Q/\Z)$ such that $\beta([w]) = q_\pm^*([w_\pm])$.  Let $w \in \mc C^3(\Mab; \Q/\Z)$ represent $[w] \in H^3(\Mab; \Q/\Z)$ and $q_\pm^*(w_\pm) \in \mc C^4(\Mab; \Z)$ represent $\beta([w]) = q_\pm^* ([w_\pm]) \in H^4(\Mab; \Z)$.  

Recall that $\beta : H^3(\Mab; \Q/\Z) \to H^4(\Mab;\Z)$ arises by applying the Snake Lemma to the commutative diagram
\beq
\label{E:Bock}
\xymatrix{
0 \ar[r] 
& \mc C^3(\Mab; \Z) \ar[r]^m \ar[d]^\delta 
& \mc C^3(\Mab; \Q) \ar[r]^r \ar[d]^\delta 
& \mc C^3(\Mab; \Q/\Z) \ar[r] \ar[d]^\delta 
& 0 \\
0 \ar[r] 
& \mc C^4(\Mab; \Z) \ar[r]^m 
& \mc C^4(\Mab; \Q) \ar[r]^r 
& \mc C^4(\Mab; \Q/\Z) \ar[r] 
& 0 
}
\eeq
of exact sequences of cochain groups, where $\delta : \mc C^3(\Mab; R) \to \mc C^4(\Mab; R)$ is the coboundary map for coefficients in $R$.  Since $\beta([w]) = q_\pm^* ([w_\pm])$, a cochain $u \in \mc C^3(\Mab; \Q)$ may thus be chosen such that 
\beq
\label{E:ruw}
r(u) = w
\ \ \text{ and } \ \ 
\delta u = m(q_\pm^*(w_\pm)).
\eeq

Notice, however, that the middle vertical map in \eqref{E:Bock} also appears in the same position in the commutative diagram
\beq
\label{E:pairdiag}
\xymatrix{
0 \ar[r] 
& \mc C^3(\Mab, M_\pm; \Q) \ar[r]^(0.55){q_\pm^*} \ar[d]^\delta 
& \mc C^3(\Mab; \Q) \ar[r]^{i_\pm^*} \ar[d]^\delta 
& \mc C^3(M_\pm; \Q) \ar[r] \ar[d]^\delta 
& 0 \\
0 \ar[r] 
& \mc C^4(\Mab, M_\pm; \Q) \ar[r]^(0.55){q_\pm^*} 
& \mc C^4(\Mab; \Q) \ar[r]^{i_\pm^*} 
& \mc C^4(M_\pm; \Q) \ar[r] 
& 0 
}
\eeq
for the pair$(\Mab, M_\pm)$.  Observe that, although $u \in \mc C^3(\Mab; \Q)$ is only a cochain, its image $i_\pm^*(u)$ under $i_\pm^* : \mc C^3(\Mab; \Q) \to \mc C^3(M_\pm; \Q)$ is a cocycle.  Indeed, from \eqref{E:ruw} and the diagram \eqref{E:bigdiag} it may be deduced that
$$
\delta(i_\pm^*(u)) = i_\pm^*(\delta u) = i_\pm^*(m(q_\pm^*(w_\pm))) = m(i_\pm^*(q_\pm^*(w_\pm))) = 0.
$$

Therefore, by applying the Snake Lemma to \eqref{E:pairdiag}, the image $\delta_\pm ([i_\pm^*(u)])$ of the class $[i_\pm^*(u)] \in H^3(M_\pm; \Q)$ under the  boundary homomorphism $\delta_\pm : H^3(M_\pm; \Q) \to H^4(\Mab, M_\pm; \Q)$ can be represented by a cocycle $c_\pm \in \mc C^4(\Mab, M_\pm; \Q)$ such that, by \eqref{E:ruw} and \eqref{E:bigdiag},
$$
q_\pm^*(c_\pm) = \delta u = m(q_\pm^*(w_\pm)) = q_\pm^*(m(w_\pm)).
$$

However, by \eqref{E:bigdiag}, the cochain map $q_\pm^* : \mc C^4(\Mab, M_\pm; \Z) \to \mc C^4(\Mab; \Z)$ is injective, implying that $c_\pm = m(w_\pm) \in \mc C^4(\Mab, M_\pm; \Q)$ and, hence, that $\delta_\pm([i_\pm^*(u)]) = [m(w_\pm)] = m([w_\pm])$.  

Since $\delta_\pm : H^3(M_\pm; \Q) \to H^4(\Mab, M_\pm; \Q)$ is an isomorphism, it thus follows from \eqref{E:bigdiag} and \eqref{E:ruw} that
\begin{align*}
r \circ \delta_\pm^{-1} \circ m ([w_\pm]) 
&= r([i_\pm^*(u)]) \\
&= i_\pm^*([r(u)]) \\
&= i_\pm^*([w]) \\
&= i_\pm^* \circ \beta^{-1} \circ q_\pm^* ([w_\pm]),
\end{align*}
as desired, where the final equality is a consequence of $\beta([w]) = q_\pm^*([w_\pm])$ and the fact that $\beta : H^3(\Mab; \Q/\Z) \to H^4(\Mab; \Z)$ is an isomorphism.
\end{proof}


\section{The linking form}
\label{S:linking}

Associated to the decomposition $\Mab = M_-\cup M_+$ of each $\Mab$ into the union of two disk-bundles with $M_- \cap M_+ = M_0$, there is a commutative braid diagram
\beq
\label{E:braid}
\resizebox{\displaywidth}{!}{
\xymatrix@=0.4cm{
{\phantom{0}} \ar[dr]^(0.4){q_-^*} \ar@(ur,ul)[rr] & 
& H^3(M_+; R) \ar[dr]^{j_+^*} \ar@(ur,ul)[rr]^{\delta_+} & 
& H^4(\Mab, M_+; R) \ar[dr]^(0.55){q_+^*} \ar@(ur,ul)[rr] & 
&  {\phantom{0}} \\
& H^3(\Mab; R)  \ar[ur]^{i_+^*} \ar[dr]^(0.55){i_-^*} & 
& H^3(M_0; R) \ar[ur]^(0.45){\partial_-} \ar[dr]^(0.5){\partial_+} & 
& H^4(\Mab; R) \ar[ur]^(0.55){i_-^*} \ar[dr]^(0.6){i_+^*} & \\
{\phantom{0}} \ar[ur]^(0.4){q_+^*} \ar@(dr,dl)[rr]& 
& H^3(M_-; R) \ar[ur]^{j_-^*} \ar@(dr,dl)[rr]_{\delta_-}  & 
& H^4(\Mab, M_-; R) \ar[ur]^(0.5){q_-^*} \ar@(dr,dl)[rr]& 
&  {\phantom{0}}
}}
\eeq
with coefficients in $R \in \{\Z, \Q, \Q/\Z\}$, where each braid is the long exact sequence of a pair.  In particular, the isomorphisms $f_\pm^* : H^j(\Mab, M_\pm; R) \to H^j(M_\mp, M_0; R)$ given by excision are being used implicitly and the homomorphism $\partial_\pm : H^3(M_0;R) \to H^4(\Mab, M_\mp; R)$ corresponds to the boundary homomorphism in the long exact sequence for the pair $(M_\pm, M_0)$.

Furthermore, given the projection $\tau : \Mab \to [-1,1]$ discussed in Section \ref{SS:family}, observe that the inclusion of the submanifold $\tau^{-1}[0,1] \In \Mab$ with boundary $\tau^{-1}\{0\}$ into the disk-bundle $M_+$ induces a homotopy equivalence $(\tau^{-1}[0,1], \tau^{-1}\{0\}) \to (M_+, M_0)$.  Therefore, Poincar\'e duality holds for $(M_+, M_0)$ just as for compact, orientable manifolds with boundary.  In particular, if $[M_+] \in H_7(M_+, M_0;\Z)$ is a fundamental class, then
$$
\frown [M_+] : H^k(M_+, M_0; R) \to H_{7-k}(M_+;R) \,;\, \alpha \mapsto \alpha \frown [M_+]
$$
is an isomorphism for all $k$.  An analogous argument works for $(M_-, M_0)$.

Let $[M] \in H_7(\Mab)$ be a fundamental class of $\Mab$.  Then $(q_-)_* [M]$ is a fundamental class for the pair $(\Mab, M_-)$ and, by excision, there is a fundamental class $[M_+] \in H_7(M_+, M_0)$ for the pair $(M_+, M_0)$ such that $(f_-)_*[M_+] = (q_-)_* [M]$.

Let $x_\pm \in H^3(M_\pm; \Z) = \Z$ be the generators used in \eqref{E:maps}.  By the Universal Coefficient Theorem, together with Table \ref{table:cohomgps}, $H^3(M_+; \Z)$ is naturally isomorphic to $\Hom(H_3(M_+), \Z)$.  Therefore, by Poincar\'e duality, a generator $\gamma_- \in H^4(\Mab, M_-; \Z) = \Z$ may be chosen such that $f_-^*(\gamma_-) \in H^4(M_+, M_0; \Z)$ is a generator and the generator $(f_-^*(\gamma_-)) \!\frown\! [M_+] \in H_3(M_+)$ is dual to $x_+$.

By exactness and since $H^4(\Mab; \Z) = \Z_{|n|}$, the boundary homomorphisms $\delta_\pm : H^3(M_\pm; \Z) = \Z \to H^4(\Mab, M_\pm; \Z) = \Z$ are given, up to sign, by multiplication by $n$.  Therefore, a generator $\gamma_+ \in H^4(\Mab, M_+; \Z)$ can be chosen such that $\delta_+(x_+) = n \gamma_+$.  Moreover, since $m : H^4(\Mab, M_+; \Z) \to H^4(\Mab, M_+; \Q)$ is injective and $\delta_+ : H^3(M_+; \Q) \to H^{4}(\Mab, M_+; \Q)$ is an isomorphism, it follows from \eqref{E:cohomdiag} that
\beq
\label{E:gens}
\delta_+^{-1} \circ m (\gamma_+) = \frac{1}{n} \, m(x_+) \in H^3(M_+; \Q).
\eeq
Now, since $q_\pm^* : H^4(\Mab, M_\pm; \Z) \to H^4(\Mab; \Z)$ are surjective, a generator of $H^4(\Mab; \Z)$ can be defined by $\one := q_+^*(\gamma_+)$.  This is the generator mentioned in Theorem \ref{T:thmB} and, furthermore, there is some $\lambda \in \Z$ such that $\lambda$ mod $|n|$ is a unit in $\Z_{|n|}$ and such that $q_-^*(\lambda \gamma_-) = \one$.

\begin{prop}
\label{P:LF}
With the notation above, the linking form 
$$
\lk : H^4(\Mab; \Z) \ox H^4(\Mab; \Z) \to \Q/\Z
$$
is given by $\lk(x \one, y \one) = \frac{\lambda xy}{n} \!\! \mod 1$.
\end{prop}

\begin{proof}
By bilinearity, only $\lk(\one, \one)$ needs to be computed.  By \eqref{E:LF},
\begin{align*}
\lk(\one, \one) &= \< \beta^{-1}(\one) \smile \one, [M]\> \\
&= \< \lambda \, \beta^{-1}(\one) \smile (q_-^*(\gamma_-)), [M] \> \\
&= \< \lambda \, q_-^*(\beta^{-1}(\one) \smile \gamma_-), [M] \>,
\end{align*}
where the last equality follows from \cite[page 251]{Sp}, since $q_- : (\Mab, \emptyset) \to (\Mab, M_\pm)$ is induced by the identity map on $\Mab$.  By naturality of the Kronecker pairing, it now follows that
\begin{align*}
\lk(\one, \one) &= \< \lambda \, \beta^{-1}(\one) \smile \gamma_-, (q_-)_* [M] \> \\
&= \< \lambda \, \beta^{-1}(\one) \smile \gamma_-, (f_-)_* [M_+] \> \\
&= \< \lambda \, f_-^*(\beta^{-1}(\one) \smile \gamma_-), [M_+] \> \\
&= \< \lambda \, i_+^*(\beta^{-1}(\one)) \smile f_-^*(\gamma_-), [M_+] \>,
\end{align*}
where the last equality again follows from \cite[page 251]{Sp}, since the map $f_- : (M_+, M_0) \to (\Mab, M_-)$ is induced by the inclusion $i_+ : M_+ \to \Mab$.  Now, by Proposition \ref{P:Bockstein} and \eqref{E:gens},
\begin{align*}
i_+^*(\beta^{-1}(\one)) &= i_+^* \circ \beta^{-1} \circ q_+^*(\gamma_+) \\
&= r \circ \delta_+^{-1} \circ m(\gamma_+) \\
&= r \left(\frac{1}{n} \, m(x_+) \right).
\end{align*}
Therefore, by naturality with respect to the inclusion $m : \Z \to \Q$ and the reduction $r : \Q \to \Q/\Z$, it follows that
\begin{align*}
\lk(\one, \one) &= \left\<\lambda \, r \left(\frac{1}{n} \, m(x_+) \right) \smile f_-^*(\gamma_-), [M_+] \right\> \\
&= r \left( \frac{\lambda}{n} \, m \left( \< x_+ \smile f_-^*(\gamma_-), [M_+] \> \right) \right) \\
&= r \left( \frac{\lambda}{n} \, m \left( \< x_+, f_-^*(\gamma_-) \frown [M_+] \> \right) \right) \\
&= r \left( \frac{\lambda}{n} \right) \\
&= \frac{\lambda}{n} \!\! \mod 1,
\end{align*}
as desired, where the second-last equality follows since $(f_-^*(\gamma_-)) \frown [M_+] \in H_3(M_+)$ is dual to $x_+ \in H^3(M_+; \Z)$.
\end{proof}

Therefore, in order to prove Theorem \ref{T:thmB}, it remains only to determine the value of $\lambda \in \Z$ in the formula for the linking form given in Proposition \ref{P:LF}.  To this end, it is necessary to first introduce two further bases, $\{u_1, u_2\}$ and $\{w_1, w_2\}$, for $H^3(M_0; \Z) = \Z \oplus \Z$, in addition to the basis $\{v_1, v_2\}$ used in \eqref{E:maps}.  Recall from \eqref{E:free} that 
$$
\gcd(a_1, a_2 \pm a_3) = 1 = \gcd(b_1, b_2 \pm b_3).
$$
Hence, there exist $e_0, e_1, f_0, f_1 \in \Z$ such that
$$
e_1 \, a_1^2 + e_0 \left(\tfrac{a_2^2 - a_3^2}{8} \right) = 1 
\ \ \text{ and } \ \ 
f_1 \, b_1^2 + f_0 \left(\tfrac{b_2^2 - b_3^2}{8} \right) = 1.
$$
Therefore, as each of the elements $j_-^*(x_-) = \frac{1}{8}(a_2^2 - a_3^2) \, v_1 + a_1^2 \, v_2$ and $j_+^*(x_+) = -\frac{1}{8}(b_2^2 - b_3^2) \, v_1 - b_1^2 \, v_2$ is a generator of $H^3(M_0;\Z)$, the two new bases can be defined via
$$
u_1 := j_-^*(x_-), 
\qquad  
u_2 := - e_1 \, v_1 + e_0 \, v_2
$$
and 
$$
w_1 := j_+^*(x_+),
\qquad 
w_2 := \ve (- f_1 \, v_1 + f_0 \, v_2),
$$
where $\ve \in \{\pm 1\}$ is such that $\delta_-(x_-) = \ve n \gamma_-$.  Define, in addition, the integers
$$
\kappa := f_1 \, a_1^2 + f_0 \left(\tfrac{a_2^2 - a_3^2}{8} \right)
\ \ \text{ and } \ \ 
\rho := e_1 \, b_1^2 + e_0 \left(\tfrac{b_2^2 - b_3^2}{8} \right),
$$
for which the following congruence identities hold:
\beq
\begin{split}
\label{E:cong}
a_1^2 \, \rho &\equiv b_1^2 \!\! \mod n, \\
b_1^2 \, \kappa &\equiv a_1^2 \!\! \mod n, \\
\frac{1}{8}(a_2^2 - a_3^2) \, \rho &\equiv \frac{1}{8}(b_2^2 - b_3^2) \!\! \mod n, \\
\frac{1}{8}(b_2^2 - b_3^2) \, \kappa &\equiv \frac{1}{8}(a_2^2 - a_3^2) \!\! \mod n.
\end{split}
\eeq
Observe, finally, that the basis element $u_2$ can be written in terms of the basis $\{w_1, w_2\}$ as
\beq
\label{E:u2}
u_2 = (e_1 \, f_0 - e_0 \, f_1)\, w_1 + \ve \rho \, w_2.
\eeq

It is now possible to complete the proof of Theorem \ref{T:thmB}.

\begin{thm}
\label{T:PfthmB}
With the notation above, the linking form $\lk : H^4(\Mab; \Z) \ox H^4(\Mab; \Z) \to \Q/\Z$ is given by 
$$
\lk(x \one, y \one) = \pm \frac{\rho \, xy}{n} \!\! \mod 1.
$$  
Alternatively, with respect to the generator $\one' := \kappa \, \one$, the linking form is given by $\lk(x \one', y \one') = \pm \frac{\kappa \, xy}{n} \!\! \mod 1$.
\end{thm}

\begin{proof}
From exactness and commutativity in the braid diagram \eqref{E:braid}, the following identities hold:  
$$
\partial_\pm \circ j_\pm^* = 0
\ \ \text{ and } \ \
\partial_\mp \circ j_\pm^* = \delta_\pm.
$$
Now, recall that $\delta_+(x_+) = n \, \gamma_+$ and $\delta_-(x_-) = \ve n \gamma_-$ for some $\ve \in \{\pm 1\}$.  Therefore, it is a simple calculation to show that the homomorphisms $\partial_\pm : H^3(M_0; \Z) \to H^4(\Mab, M_\mp; \Z)$ are given by
$$
\partial_-(v_1) = -a_1^2 \, \gamma_+, 
\qquad
\partial_-(v_2) = \frac{a_2^2 - a_2^3}{8} \, \gamma_+
$$
and
$$
\partial_+(v_1) = - \ve b_1^2 \, \gamma_-, 
\qquad
\partial_+(v_2) = \ve \frac{b_2^2 - b_2^3}{8} \, \gamma_-,
$$
respectively.  From the definition of the bases $\{u_1, u_2\}$ and $\{w_1, w_2\}$, it now follows that
$$
\partial_-(u_1) = 0,
\qquad
\partial_-(u_2) = \gamma_+,
\ \ \text{ and } \ \
\partial_-(w_2) = \ve \kappa \, \gamma_+,
$$
while
$$
\partial_+(w_1) = 0,
\qquad
\partial_+(w_2) = \gamma_-,
\ \ \text{ and } \ \
\partial_+(u_2) = \ve \rho \, \gamma_-.
$$
Therefore, by \eqref{E:braid} and \eqref{E:u2},
\begin{align*}
\lambda \, q_-^* (\gamma_-) &= \one \\
&= q_+^*(\gamma_+) \\
&= q_+^*(\partial_-(u_2)) \\
&= q_-^*(\partial_+(u_2)) \\
&= q_-^*(\partial_+((e_1 \, f_0 - e_0 \, f_1)\, w_1 + \ve \rho \, w_2)) \\
&= \ve \rho \, q_-^*(\partial_+(w_2)) \\
&= \ve \rho \, q_-^*(\gamma_-) \in H^4(\Mab; \Z) = Z_{|n|},
\end{align*}
from which it immediately follows that $\lambda \equiv \ve \rho$ mod $n$, as desired.  The final statement in the theorem follows from a direct calculation showing that $\kappa \, \rho \equiv 1$ mod $n$, since this implies that $\one' = \kappa \one$ is a generator of $H^4(\Mab; \Z) = \Z_{|n|}$.
\end{proof}


\section{Some elementary number theory}
\label{S:numth}

As a simple corollary of Theorem \ref{T:PfthmB}, it turns out that any $\Mab$ with $\gcd(a_1, b_1) = 1$ and satisfying \eqref{E:finite} has standard linking form.  Such manifolds include, of course, all $\sph^3$-bundles over $\sph^4$ with non-trivial $H^4$, as described by Grove and Ziller \cite{GZ}, which are well known to have standard linking form \cite{CE}.

\begin{lem}
\label{L:stdLF}
Every $\Mab$  with $\gcd(a_1, b_1) = 1$ and satisfying \eqref{E:finite} is homotopy equivalent, hence PL-homeomorphic, to an $\sph^3$-bundle over $\sph^4$.
\end{lem}

\begin{proof}
Suppose that $\Mab$ has $\gcd(a_1, b_1) = 1$.  Then, by the definition of $n$, 
$$
\gcd(a_1, n) = 1 = \gcd(b_1, n),
$$
that is, $a_1$ mod $n$ and $b_1$ mod $n$ are units in $\Z_{|n|}$.  Therefore, $a_1 \one$ and $b_1 \one$ are generators of $H^4(\Mab; \Z) = \Z_{|n|}$.  In particular, by \eqref{E:cong} and Theorem \ref{T:PfthmB},
\begin{align*}
\lk(a_1 \one, a_1  \one) &= \pm \frac{a_1^2 \, \rho}{n} \!\! \mod 1 \\
&= \pm \frac{b_1^2}{n} \!\! \mod 1.
\end{align*}
Now, by the definition of standard linking form in Section \ref{SS:link} and Theorem \ref{T:KS}, the result follows.
\end{proof}

As a consequence of Lemma \ref{L:stdLF}, to have any hope of obtaining manifolds $\Mab$ with non-standard linking form, it is necessary to assume that $\gcd(a_1, b_1) \neq 1$.  In particular, this implies that there is some prime $p$ dividing $n$ such that $p^2$ also divides $n$.  Therefore, as in Example \ref{Eg:nonst}, whenever $n$ is not divisible by $p^2$ for all prime divisors $p$ of $n$, there is the possibility of finding manifolds with non-standard linking form which cannot be described as a manifold $\Mab$.  Hence, the manifolds $\Mab$ do not realise all $2$-connected $7$-manifolds with $H^4$ finite cyclic.

Returning to the search for manifolds $\Mab$ with non-standard linking form, the following simple observation will prove useful.

\begin{lem}
\label{L:primes}
Suppose $d \in \N$ divides  $n \in \N$ and that $k \in \Z$ is not a square mod $d$.  Then $k \in \Z$ is not a square mod $n$.
\end{lem}

\begin{proof}
Suppose that there is some $l \in \Z$ such that $k \equiv l^2$ mod $n$.  Then it is clear that $k \equiv l^2$ mod $d$, a contradiction. 
\end{proof}

It now turns out that it is reasonably straightforward to find examples of manifolds $\Mab$ with non-standard linking form.  To avoid that the computations to follow become unnecessarily complicated, let
$$
a_0 := \frac{a_2^2 - a_3^2}{8},
\qquad
b_0 := \frac{b_2^2 - b_3^2}{8}.
$$
With this notation, 
\beq
\label{E:conds}
\begin{split}
n &= a_1^2 \, b_0 - a_0\, b_1^2, \\
1 &= e_1 \, a_1^2 + e_0 \, a_0, \\  
1 &= f_1 \, b_1^2 + f_0 \, b_0.
\end{split}
\eeq

Recall that, for $p$ an odd prime and $x \in \Z$, the \emph{Legendre symbol} $\bigl(\frac{x}{p} \bigr)$ is defined via
$$
\Bigl(\frac{x}{p} \Bigr) = 
\begin{cases}
\phantom{-}1, & \text{ if $x$ is a square mod $p$ and $x \not\equiv 0$ mod $p$},\\
-1, & \text{ if $x$ is not a square mod $p$},\\
\phantom{-}0, & \text{ if $x \equiv 0$ mod $p$}.
\end{cases}
$$
The Legendre symbol has the following properties:
\beq
\label{E:legendre}
\begin{split}
\Bigl(\frac{x}{p} \Bigr) &= \Bigl(\frac{y}{p} \Bigr), \ \ \text{ if } x \equiv y \!\! \mod p \,; \\
\Bigl(\frac{xy}{p} \Bigr) &= \Bigl(\frac{x}{p} \Bigr) \Bigl(\frac{y}{p} \Bigr).
\end{split}
\eeq
The first supplement to the law of quadratic reciprocity states that 
\beq
\label{E:-1}
\Bigl(\frac{-1}{p} \Bigr) = 1 \ \text{  if and only if } p \equiv 1 \!\! \mod 4,
\eeq
that is, $-1$ is a square if and only if $p \equiv 1$ mod $4$.

\begin{thm}
\label{T:egs}
Suppose $\Mab$ satisfies \eqref{E:finite} and that there is a prime $p \equiv 1$ mod $4$ such that $p$ divides $\gcd(a_1, b_1)$.  If $a_0$ is not a square mod $p$ and $b_0$ is a square mod $p$, then $\Mab$ has non-standard linking form and, hence, is not even homotopy equivalent to an $\sph^3$-bundle over $\sph^4$. 
\end{thm}

\begin{proof}
By Theorem \ref{T:PfthmB}, there is a generator $\one \in H^4(\Mab; \Z) = \Z_{|n|}$ such that $\lk(\one, \one) = \pm \frac{\rho}{n}$ mod $1$, with $\rho = e_1 \, b_1^2 + e_0 \, b_0$.  On the other hand, by \eqref{E:conds},
$e_0 \, a_0 \equiv 1 \!\! \mod p$.  Since $1$ is obviously a square mod $p$, \eqref{E:legendre} implies that
$$
1 = \Bigl(\frac{1}{p} \Bigr) 
= \Bigl(\frac{e_0 \, a_0}{p} \Bigr) 
= \Bigl(\frac{e_0}{p} \Bigr) \Bigl(\frac{a_0}{p} \Bigr) 
= - \Bigl(\frac{e_0}{p} \Bigr),
$$
because $\bigl(\frac{a_0}{p} \bigr) = -1$, by hypothesis.  That is, $\bigl(\frac{e_0}{p} \bigr) = -1$.

Therefore, since $p \equiv 1$ mod $4$ was assumed to divide $b_1$,
\begin{align*}
\Bigl(\frac{\pm \rho}{p} \Bigr) 
&= \Bigl(\frac{\pm (e_1 \, b_1^2 + e_0 \, b_0)}{p} \Bigr) \\
&= \Bigl(\frac{ \pm e_0 \, b_0}{p} \Bigr) \\
&= \Bigl(\frac{\pm 1}{p} \Bigr) \Bigl(\frac{e_0}{p} \Bigr) \Bigl(\frac{b_0}{p} \Bigr) \\
&= -1,
\end{align*}
where the final equality follows from \eqref{E:-1}, $\bigl(\frac{e_0}{p} \bigr) = -1$ and the hypothesis that $b_0$ is a square mod $p$.

Hence, $\pm \rho$ is not a square mod $p$ and, by Lemma \ref{L:primes}, it follows that $\pm \rho$ is not a square mod $n$.  However, since $\pm \rho$ is a unit in $\Z_{|n|}$ (by the proof of Theorem \ref{T:PfthmB}), this implies that $\Mab$ has a non-standard linking form, as desired.
\end{proof}

Explicit examples satisfying the hypotheses of Theorem \ref{T:egs} are plentiful.  Indeed, note that, by a simply counting argument, for any prime $p \equiv 1$ mod $4$ there must be a pair $m$, $m+1$, $m \in \{1, \dots, p-2\}$, of consecutive integers such that $\bigl( \frac{m}{p} \bigr) = -1$ and $\bigl( \frac{m+1}{p} \bigr) = 1$.

\begin{cor}
Let $p \equiv 1$ mod $4$ be an odd prime.  If $m \in \{1, \dots, p-2\}$ is such that $\bigl( \frac{m}{p} \bigr) = -1$ and $\bigl( \frac{m+1}{p} \bigr) = 1$, then $a_1 = b_1 = p$, $|a_2| = 2m - 1$, $|a_3| = |b_2| = 2m +1$ and $|b_3| = 2m+3$ define a manifold $\Mab$ with non-standard linking form.
\end{cor}

\begin{proof}
Observe first that some choice of signs for $\pm(2m-1)$, $\pm(2m+1)$ and $\pm(2m+3)$ yields integers $\equiv 1$ mod $4$.  Furthermore, $a_2^2 - a_3^2 = -8m$ and $b_2^2 - b_3^2 = -8(m+1)$ are, by definition, prime to $p = a_1 = b_1$.  Therefore, the freeness conditions \eqref{E:free} are satisified and $\ul a$, $\ul b$ define a manifold $\Mab$.  Moreover, $n = -p^2$, so that $H^4(\Mab; \Z) = \Z_{p^2}$.  

Now $a_0 = -m$ and $b_0 = -(m+1)$.  Thus, by the hypotheses on $m  \in \{1, \dots, p-2\}$, Theorem \ref{T:egs} implies that $\Mab$ has non-standard linking form.
\end{proof}


\bibliographystyle{alpha}

\end{document}